\documentclass[12pt,letterpaper]{article}
\usepackage{amsmath}
\usepackage{amsfonts}
\usepackage{amsthm}
\usepackage{amssymb}
\usepackage{hyperref}
\usepackage{verbatim}

\DeclareMathOperator{\lcm}{lcm}

\newtheorem{prop}{Proposition}
\newtheorem{corollary}[prop]{Corollary}

\newtheorem{lemma}[prop]{Lemma}

\newcommand{\seqnum}[1]{\href{http://oeis.org/#1}{\underline{#1}}}

\begin{document}
\title{There are no Coincidences}
\author{Tanya Khovanova}
\maketitle

\begin{abstract}
This paper is inspired by a seqfan post by Jeremy Gardiner. The post listed nine sequences with similar parities. In this paper I prove that the similarities are not a coincidence but a mathematical fact.
\end{abstract}

\section{Introduction}
Do you use your Inbox as your implicit to-do list? I do. I delete spam and cute cats. I reply, if I can. I deal with emails related to my job. But sometimes I get a message that requires some thought or some serious time to process. It sits in my Inbox. There are 200 messages there now. This method creates a lot of stress. But this paper is not about stress management. Let me get back on track.

I was trying to shrink this list and discovered an email I had completely forgotten about. The email was sent in December 2008 to the seqfan list by Jeremy Gardiner \cite{GD}.

The email starts:

\begin{quote}
It strikes me as an interesting coincidence that the following sequences
appear to have the same underlying parity.
\end{quote}

Then he lists six sequences with the same parity:

\begin{itemize} 
\item \seqnum{A128975} The number of unordered three-heap P-positions in Nim.
\item \seqnum{A102393}, A wicked evil sequence.
\item \seqnum{A029886}, Convolution of the Thue-Morse sequence with itself.
\item \seqnum{A092524}, Binary representation of $n$ interpreted in base $p$, where $p$ is the smallest prime factor of $n$.
\item \seqnum{A104258}, Replace $2^i$ with $n^i$ in the binary representation of $n$.
\item \seqnum{A061297}, $a(n) = \sum_{r = 0}^n \lcm (n,n-1, n-2, \ldots, n-r+1)/ \lcm ( 1,2,3,\ldots,r)$.
\end{itemize}

and three sequences with the complementary parity:

\begin{itemize} 
\item \seqnum{A093431}, $a(n)=\sum_{k=1}^n (\lcm(n,n-1,\ldots,n-k+2,n-k+1)/\lcm(1,2,\ldots,k))$.
\item \seqnum{A003071}, Maximal number of comparisons for sorting $n$ elements by list merging.
\item \seqnum{A122248}, Partial sums of $a(n)$, where $a(n) = a(\lfloor n/2\rfloor) + \lfloor n/2 \rfloor$ with $a(1) = 1$. 
\end{itemize}

In this paper I will prove that these coincidences are in fact not coincidences. And my Inbox is now reduced to 199 messages.

But first let me start with definitions.

\section{The Thue-Morse sequence}

Let us start with evil and odious numbers introduced by John Conway. A non-negative integer is called \textit{evil} if the number of ones in its binary expansion is even and it is called \textit{odious} otherwise. 

The sequence of evil numbers is \seqnum{A001969}: 0, 3, 5, 6, 9, 10, 12, 15, 17, 18, 20, 23, 24, $\ldots$, starting from index 1. We will denote the sequence as $e(n)$. 
The sequence of odious numbers is \seqnum{A000069}: 1, 2, 4, 7, 8, 11, 13, 14, 16, 19, 21, 22, 25, 26, $\ldots$, starting from index 1. We will denote the sequence as $o(n)$.

Let $s_2(n)$ denote the binary weight of $n$: the number of ones in the binary expansion of $n$. Thus $n$ is evil if $s_2(n) \equiv 0 \pmod{2}$, and $n$ is odious if $s_2(n) \equiv 1 \pmod{2}$.

The Thue-Morse sequence, $t(n)$, is the parity of the sum of the binary digits of $n$, which is also call the \textit{perfidy} of $n$. It is \seqnum{A010060} (starting with index 0): 0, 1, 1, 0, 1, 0, 0, 1, $\ldots$. Perfidy is to parity as evil is to even and odious is to odd. By definition, $t(n)=1$ if $n$ is odious, and 0 otherwise. In other words, $t(n) = s_2(n) \pmod{2}$. Or, $t(n)$ is the characteristic function of odious numbers.

The Thue-Morse sequence has many interesting properties:

\begin{itemize}
\item recursive definition: $t(2n) = t(n)$ and $t(2n+1) = 1- t(n)$, where $t(0) = 0$.
\item fractal property: the sequence $t(n)$ is a fixed point of the morphism $0 \rightarrow 0,1$ and $1 \rightarrow 1,0$.
\item cube-free: the sequence $t(n)$ does not contain three consecutive identical blocks. In particular it does not contain $0,0,0$ nor $1,1,1$. 
\end{itemize}

We will also need a different version of the Thue-Morse sequence, \seqnum{A010059}: 1, 0, 0, 1, 0, 1, 1, 0, 0, 1, 1, 0, 1, $\ldots$: the logical negation of the Thue-Morse sequence. Let us denote this sequence as $\bar{t}(n)$: $\bar{t}(n) + t(n) = 1$. The sequence is the characteristic function of the evil numbers.

The properties of the sequence $\bar{t}(n)$ are exactly the same as the properties of the sequence $t(n)$ except that it starts with 1 rather than 0.

\section{The Sequence}

This is the sequence we will be comparing against. This is the main object of this paper so I denote it as $m(n)$. This sequence is the alternate merging of $\bar{t}(n)$ and the all-zero sequence. Namely $m(2k+1) = 0$ and $m(2k) = \bar{t}(k)$.

The sequence starting from the index 0: 1, 0, 0, 0, 0, 0, 1, 0, 0, 0, 1, 0, 1, 0, 0, 0, 0, 0, 1, 0, 1, 0, 0, 0, 1, $\ldots$. In the database it is represented as \seqnum{A228495}$(n+1)$, the characteristic function of the odd odious numbers shifted by 1. Let us see that this is the same definition. If $n$ is even, then $n$ is not an odd odious number, so \seqnum{A228495}$(n) = 0 = m(n-1)$. If $n=2k+1$, then $n$ is odious if and only if $k$ is evil. As $\bar{t}$ is the characteristic function of evil numbers we get \seqnum{A228495}$(2k+1) = \bar{t}(k) = m(2k)$.

Now we describe the properties of this sequence:

\begin{itemize}
\item recursive definition: $m(2n+1) = 0$, $m(4n) = \bar{t}(2n) = \bar{t}(n) = m(2n)$, and $m(4n+2) = \bar{t}(2n+1) = 1 - \bar{t}(n) = 1 - m(2n)$, where $m(0) = 1$.
\item fractal property: the sequence $m(n)$ is a fixed point of the morphism $0,0 \rightarrow 0,0,1,0$ and $1,0 \rightarrow 1,0,0,0$.
\item pattern-free: the sequence $m(n)$ does not contain more than five consecutive 0s or more than one consecutive 1. It does not contain three $1,0$ blocks in a row.
\end{itemize}

We will use the following property of the sequence $m(n)$, which follows immediately from above.

\begin{corollary}\label{thm:cor1}
For an even number $n$, $m(n)$ is the parity of $s_2(n/2)-1$.
\end{corollary}

\section{Resolving coincidences}

Now I will resolve Jeremy Gardiner's list one by one.

\subsection{\seqnum{A128975} The number of unordered three-heap P-positions in Nim}         

\seqnum{A128975}, The number of unordered three-heap P-positions in Nim: 0, 0, 0, 0, 0, 1, 0, 0, 0, 1, 0, 1, 0, 4, 0, 0, 0, $\ldots$, starting with index 1. In other words this sequence is the number of unordered triples of non-zero integers $(a,b,c)$ with $a+b+c=n$, whose bitwise XOR is zero. 

Notice that $n$ must be even for a P-position to exist: $\seqnum{A128975}(2k+1) = 0$.

\begin{lemma}
The parity of \seqnum{A128975}$(n)$ is $m(n)$.
\end{lemma}

\begin{proof}
In my recent paper with Joshua Xiong \cite{KX} we showed that the number of ordered three-heap P-positions in Nim where an empty heap is allowed is \seqnum{A048883}($n$) $ = 3^{s_2(n/2)}$, where the total number of counters is an even number $n$.

These sequences \seqnum{A128975} and \seqnum{A048883} are related to each other in a simple manner. There are no P-positions of type $(a,a,a)$. Also $(a,a,0)$ is a P-position for any $a$. That means unordered P-positions $(a,b,c)$ that are counted in \seqnum{A128975} are composed of three distinct positive numbers. So we need to multiply this number by 6 to count ordered sets of distinct numbers and add 3 P-positions that are permutations of $(n/2,n/2,0)$. Long story short: $\seqnum{A048883}(n) =  6 \cdot \seqnum{A128975}(n) + 3$. Therefore, $\seqnum{A128975}(n) =  (3^{s_2(n/2)} - 3)/6= (3^{s_2(n/2)-1} - 1)/2$. 

We know that $3^{2k} \equiv 1 \pmod{4}$ and $3^{2k} \equiv 3 \pmod{4}$. That means the parity of \seqnum{A128975} is the same as the parity of $s_2(n/2)-1$ for even $n$, which by Corollary~\ref{thm:cor1} is $m(n)$.
\end{proof}

\subsection{\seqnum{A102393}, A wicked evil sequence}

\seqnum{A102393}, A wicked evil sequence: 1, 0, 0, 4, 0, 6, 7, 0, 0, 10, 11, 0, 13, 0, 0, 16, 0, 18, 19, $\ldots$ starting from index 0.

Another definition of this sequence states that evil numbers plus 1 appear in positions indexed by evil numbers. That means $\seqnum{A102393}(n) = 0$ if $n$ is odious, and $\seqnum{A102393}(e(k)) = e(k)+1$. 

\begin{lemma}
The parity of \seqnum{A102393}$(n)$ is $m(n)$.
\end{lemma}

\begin{proof}
For any $n$, $\seqnum{A102393}(n)$ is either 0 or $n+1$. That means $\seqnum{A102393}(n)$ is even for odd $n$. If $n=2k$, then $\seqnum{A102393}(n)$ is the non-perfidy (the logical negation of the perfidy) of $n$. In this case it is the same as the non-perfidy of $k$, which is the same as the parity of $s_2(n/2)-1$.
\end{proof}

\subsection{\seqnum{A029886}, Convolution of the Thue-Morse sequence with itself}

\seqnum{A029886}, Convolution of the Thue-Morse \seqnum{A001285} sequence with itself: 1, 4, 8, 10, 12, 14, 15, 16, 22, 24, 23, 26, 29, $\ldots$ starting from index 0. In the definition of this sequence another variation of the Thue-Morse sequence is used, which is the same as $\bar{t}(n)$ with zeros replaced by twos: 1, 2, 2, 1, 2, 1, 1, 2, 2, 1, 1, $\ldots$.

From the point of view of parity this sequence is equivalent to the convolution of $\bar{t}(n)$ with itself, which is the sequence 1, 0, 0, 2, 0, 2, 3, 0, 2, 4, 3, 2, 5, 2, 2, 8, 2, 4, 7 $\ldots$. This is now sequence \seqnum{A247303}: $a(n) = \sum_{i=0}^n \bar{t}(i) \bar{t}(n-i)$.

\begin{lemma}
The parity of \seqnum{A029886}$(n)$ is $m(n)$.
\end{lemma}

\begin{proof}
If $n$ is odd we can split the sum into pairs of equal numbers: $\bar{t}(i) \bar{t}(n-i) = \bar{t}(n-i)\bar{t}(i)$. Thus the result is zero. If $n=2k$, then the sum is $\bar{t}(k)\bar{t}(k) = \bar{t}(k) = m(n)$.
\end{proof}

The sequence \seqnum{A247303} has the same parity as \seqnum{A029886}, so it could be added to the list.

\begin{corollary}
The parity of \seqnum{A247303}$(n)$ is $m(n)$.
\end{corollary}

\subsection{\seqnum{A092524}, Binary representation of $n$ interpreted in base $p$, where $p$ is the smallest prime factor of $n$}

\seqnum{A092524}, Binary representation of $n$ interpreted in base $p$, where $p$ is the smallest prime factor of $n$: 1, 2, 4, 4, 26, 6, 57, 8, 28, 10, 1343, 12, 2367, 14, 40, $\ldots$ starting from index 1. In other words, the sequence can be constructed by replacing powers of 2 in the binary representation of $n$ with powers of $p$, where $p$ is the smallest prime factor of $n$.

\begin{lemma}\label{thm:binarynp}
The parity of \seqnum{A092524}$(n)$ is $m(n+1)$.
\end{lemma}

\begin{proof}
If $n$ is even, then the smallest prime factor is 2 and \seqnum{A092524}$(n) = n$, and the parity of \seqnum{A092524}$(n)$ is 0 = $m(n+1)$. If $n$ is odd, the smallest prime factor is odd, and the result is the sum of $s_2(n)$ odd numbers. The parity of this number is the perfidy of $n$; that is, the parity of $s_2(n)$, which is equal to $m(n+1)$.
\end{proof}

\subsection{\seqnum{A104258}, Replace $2^i$ with $n^i$ in the binary representation of $n$}

\seqnum{A104258}, Replace $2^i$ with $n^i$ in the binary representation of $n$: 1, 2, 4, 16, 26, 42, 57, 512, 730, 1010, 1343, 1872, $\ldots$ starting from index 1.

\begin{lemma}
The parity of \seqnum{A104258}$(n)$ is $m(n+1)$.
\end{lemma}

\begin{proof}
The parity of $n^i$ is the same as the parity of $n$'s smallest prime factor. So this proof is similar to the proof of Lemma~\ref{thm:binarynp}.
\end{proof}

\subsection{\seqnum{A061297}, $a(n) = \sum_{r = 0}^n \lcm (n,n-1, n-2, \ldots, n-r+1)/ \lcm ( 1,2,3,\ldots,r)$}

\seqnum{A061297}, $a(n) = \sum_{r = 0}^n \lcm (n,n-1, n-2, \ldots, n-r+1)/ \lcm ( 1,2,3,\ldots,r)$: 1, 2, 4, 8, 14, 32, 39, 114, 166, 266, 421, 1608, $\ldots$ starting from index index 0.

Here we assume that $\lcm() = 1$; that is, the sum starts with 1.

Let us begin with calculating the parity of $\lcm (n,n-1, n-2, \ldots, n-r+1)/ \lcm (1,2,3,\ldots,r)$. For this we need to know the highest power of 2 that divides $\lcm (n,n-1, n-2, \ldots, n-r+1)$ and $\lcm (1,2,3,\ldots,r)$. The highest power of 2 that divides $k$ is called the \textit{2-adic value} of $k$.

The 2-adic value of $\lcm (1,2,3,\ldots,r)$ is the highest power of 2 contained in the set $\{1,2,3,\ldots,r\}$ or $\lceil \log_2 r \rceil$. The 2-adic value of $\lcm (n,n-1, n-2, \ldots, n-r+1)$ is the highest 2-adic value of elements contained in the set $\{n,n-1, n-2, \ldots, n-r+1\}$. We actually do not need to know the value itself; we just need to know if it is greater than $k=\lceil \log_2 r \rceil$. It is greater than $k$ if and only if the last $k$ digits of the binary representation of $n$ represent a number less than $r$.

\begin{lemma}
The parity of \seqnum{A061297}$(n)$ is $m(n)$.
\end{lemma}

\begin{proof}
If $n$ is odd, then $r=2k$ and $r=2k+1$ both give an odd or an even contribution to the sum. Hence, the sequence \seqnum{A061297}$(n)$ is even for odd indices. If $n$ is even, we split the sum into pairs for $r=2k$ and $r=2k+1$. For which $r$ do the numbers in a pair give different contributions? Only for $r$ that are equal to the last several binary digits of $n$, including 0. The number $n$ itself is not paired, but it contributes 1 to the sum, so we should count it too. The number of different $r$'s like that is the number of ones in the binary representation of $n$ plus 1, which is equal to $s_2(n/2)+1$.
\end{proof}

\subsection{\seqnum{A093431}, $a(n)=\sum_{k=1}^n \lcm(n,n-1,\ldots,n-k+2,n-k+1)/\lcm(1,2,\ldots,k)$}

\seqnum{A093431}, $a(n)=\sum_{k=1}^n \lcm(n,n-1,\ldots,n-k+2,n-k+1)/\lcm(1,2,\ldots,k)$: 1, 3, 7, 13, 31, 38, 113, 165, 265, 420, 1607, 1004, $\ldots$ starting from index 1.

\begin{lemma}
The parity of \seqnum{A093431}$(n)$ is $1-m(n+1)$.
\end{lemma}

\begin{proof}
This sequence is similar to  \seqnum{A061297}; we just start from 1  rather than from 0. As the value for 1 is 1, we get $\seqnum{A093431}(n) = 1 - \seqnum{A061297}(n)$.
\end{proof}

\subsection{\seqnum{A003071}, Maximal number of comparisons for sorting $n$ elements by list merging}

\seqnum{A003071}, Maximal number of comparisons for sorting $n$ elements by list merging: 0, 1, 3, 5, 9, 11, 14, 17, 25, 27, 30, 33, 38, 41, 45, 49, 65, $\ldots$ starting from index 1.

The sequence should not be confused with \seqnum{A001855}, Maximal number of comparisons for sorting $n$ elements by binary insertion: 0, 1, 3, 5, 8, 11, 14, 17, 21, 25, 29, 33, $\ldots$ starting from index 1.

The list-merging  sort \cite{K} is defined recursively. We start with $n$ lists of size 1. We group the lists into pairs and merge them. After the first iteration we have several lists of size 2 and maybe one extra list. We group these lists into pairs and merge again. After $k$ iterations we will have several lists of size $2^k$ and maybe one more list of size not more than $2^k$. 

Merging two lists can also be defined recursively. We start by comparing the two smallest elements in each list. The smaller of the two is the first element of the resulting list. We can remove it from its list and continue recursively. The merging requires $n-1$ comparisons, where $n$ is the total number of elements in both lists.

\begin{lemma}
The parity of \seqnum{A003071}$(n)$ is $1-m(n+1)$.
\end{lemma}

\begin{proof}
Let us first count the parity of comparisons for the powers of 2: $n=2^k$. When the number of lists is divisible by 4 then merging pairs of them requires an even number of comparisons. Thus the parity of the total number of comparisons is the same as the parity of the number of comparisons in the last step, when we merge two lists of size $2^{k-1}$. This merging requires an odd number of comparisons, thus the total is odd.

Now suppose $k$ is the largest integer such that $2^k \leq n$: $n=2^k + x$, where $x< 2^k$. At the last step we merge two lists of size $2^k$ and $x$. Thus \seqnum{A003071}$(n) = \seqnum{A003071}(2^k) + \seqnum{A003071}(x) + n-1$. The parity of this number is the same as the parity of \seqnum{A003071}$(x) + x$. If $n$ is even, then the parity of \seqnum{A003071}$(n)$ is the same as the parity of \seqnum{A003071}$(x)$. Now we can invoke induction to show that for even $n$ the sequence is odd. 

Suppose $n$ is odd, then \seqnum{A003071}$(n) \equiv 1- \seqnum{A003071}(x) \pmod{2}$. This is the same recursive behavior as in our sequence $m(n+1)$. Comparing the starting numbers we finish the proof.
\end{proof}

\subsection{\seqnum{A122248}, Partial sums of $a(n)$, where $a(n) = a(\lfloor n/2\rfloor) + \lfloor n/2 \rfloor$ with $a(1) = 1$}

\seqnum{A122248}, Partial sums of $a(n)$, where $a(n) = a(\lfloor n/2\rfloor) + \lfloor n/2 \rfloor$ with $a(1) = 1$: 0, 1, 3, 5, 9, 13, 18, 23, 31, 39, 48, 57, 68, 79, 91, 103, 119, $\ldots$ starting from index 0.

The sequence $a(n)$ is \seqnum{A113474}: 1, 2, 2, 4, 4, 5, 5, 8, 8, 9, 9, $\ldots$. Notice that $a(2k+1) = a(2k)$. Consider the sequence $b(k) = a(2k)= \seqnum{A101925}(k)$. It follows that $b(k) = a(k) + k= b(\lfloor k/2\rfloor) + k$. It follows that the parity of this sequence is $\bar{t}(k)$. By the way, $\seqnum{A101925}(n) =\seqnum{A005187}(n) + 1$, where \seqnum{A005187}(n) is the highest power of 2 in $(2n)!$.

\begin{lemma}
The parity of \seqnum{A122248}$(n)$ is $1-m(n)$.
\end{lemma}

\begin{proof}
If $n$ is odd then the partial sum can be grouped into pairs of equal numbers, except for the first one, so \seqnum{A122248}$(2k+1) \equiv 1 \pmod{2} \equiv m(2k+2) \pmod{2}$. If $n$ is even, then we still can pair all but the first and the last terms. So \seqnum{A122248}$(2k) \equiv 1- \seqnum{A113474}(2k) \pmod{2} \equiv \bar{t}(k) \pmod{2} \equiv 1 - \bar{t}(k) \equiv 1- m(n) \pmod{2}.$
\end{proof}

\section{Acknowledgements}\label{sec:acknowledgements}

I am grateful to Julie Sussman, P.P.A., for editing this paper.

\bigskip
\hrule
\bigskip

\noindent 2010 {\it Mathematics Subject Classification}: Primary 11B99.

\noindent \emph{Keywords: Thue-Morse sequence, Nim, sort, evil, odious} 

\bigskip
\hrule
\bigskip

\noindent (Concerned with sequences
 \seqnum{A000069},
 \seqnum{A001285},
 \seqnum{A001855},
 \seqnum{A001969},
 \seqnum{A003071},
 \seqnum{A005187},
 \seqnum{A010059},
 \seqnum{A010060},
 \seqnum{A029886},
 \seqnum{A048883},
 \seqnum{A061297},
 \seqnum{A092524},
 \seqnum{A093431},
 \seqnum{A101925},
 \seqnum{A102393},
 \seqnum{A104258},
 \seqnum{A113474},
 \seqnum{A122248},
 \seqnum{A128975},
 \seqnum{A228495}, and
 \seqnum{A247303}.)


\begin{thebibliography}{9}


\bibitem{GD} Jeremy Gardiner, An email to seqfan list, available at: \url{http://list.seqfan.eu/pipermail/seqfan/2008-December/000411.html}

\bibitem{KX} Tanya Khovanova and Joshua Xiong, Nim Fractals, Journal of Integer Sequences, 17, 2014, Article 14.7.8.

\bibitem{K} D. E. Knuth, Art of Computer Programming, Vol. 3, Section 5.2.4.

\bibitem{OEIS} \emph{The On-Line Encyclopedia of Integer Sequences},  published electronically at \url{http://oeis.org}, 2014.

\bibitem{TM} \emph{Thue-Morse sequence}, wikipedia article, available electronically at \url{https://en.wikipedia.org/wiki/Thue%E2%80%93Morse_sequence}, 2014.

\end{thebibliography}
\end{document}